\documentclass{article}
\usepackage{amssymb}
\usepackage{amsmath}
\usepackage{latexsym}
\usepackage{amsthm}
\usepackage[all]{xy}  
\numberwithin{equation}{section}   

\newtheorem{theorem}{Theorem}[section]
\newtheorem{proposition}[theorem]{Proposition}
\newtheorem{lemma}[theorem]{Lemma}
\newtheorem{corollary}[theorem]{Corollary}

\theoremstyle{definition}
\newtheorem{definition}[theorem]{Definition}
\newtheorem{example}[theorem]{Example}
\newtheorem{remark}[theorem]{Remark}

\newcommand{\R}{{\mathbb{R}}}

\newcommand{\Z}{{\mathbb{Z}}}
\DeclareMathOperator{\id}{id}

\def \g {\mathfrak{g}}

\def \a {\mathfrak{a}}
\def \k {\mathfrak{k}}

\def \zg {\mathfrak{z}(\mathfrak{g})}
\newcommand{\Og}[1][]{\Omega^{#1}}
\newcommand{\Ogtot}[1][]{\Omega^{#1}_{\rm tot}}
\newcommand{\Otg}[1][]{\Omega^{#1}_{\rm tot}(\gm)}
\newcommand{\OtGH}[1][]{\Omega^{#1}_{\rm tot}([G\rto[i] H])}

\newcommand{\Ng}[1][]{N_{#1}\gm}

\renewcommand{\ker}{\operatorname{Ker}}
\DeclareMathOperator{\coker}{Coker}
\DeclareMathOperator{\im}{Im}

\newcommand{\gto}[1][]{\xrightarrow[]{#1}}
\newcommand{\rto}[1][]{\xrightarrow[]{#1}}

\newcommand{\toto}{\rightrightarrows}

\newcommand{\Hom}[1][]{\mathrm{Hom}_{\raise1.5ex\hbox to.1em{}#1}}
\renewcommand{\hom}[1][]{{\mathcal{H}om}_{\raise1.5ex\hbox to.1em{}#1}}

\newcommand{\Aut}[1][]{\mathrm{Aut}_{\raise1.5ex\hbox to.1em{}#1}}
\newcommand{\Out}[1][]{\mathrm{Out}_{\raise1.5ex\hbox to.1em{}#1}}

\newcommand{\ooplus}{\mathop{\bigoplus}\limits}
\newcommand{\com}{{\scriptscriptstyle\bullet}}
\newcommand{\rr}{{\mathbb R}}

\newcommand{\ba}{\begin{array}}
\newcommand{\ea}{\end{array}}

\newenvironment{anum}{
  \begin{enumerate}
  \itemindent=-3pt
  \itemsep=0pt
  
  }
  {\end{enumerate}}

\newcommand{\bnum}{\begin{enumerate}[{\rm(i)}]}
\newcommand{\enum}{\end{enumerate}}
\newcommand{\banum}{\begin{enumerate}[{\rm(a)}]}
\newcommand{\eanum}{\end{enumerate}}

\newcommand{\eq}{\begin{eqnarray}}
\newcommand{\eneq}{\end{eqnarray}}
\newcommand{\eqn}{\begin{eqnarray*}}
\newcommand{\eneqn}{\end{eqnarray*}}
\def \zz {\mathbb{Z}}
\def \ddR {d_{\rm dR}}
\def \dpa {\partial }

\def \gm {\Gamma}

\begin{document}  

\title{Cohomology of  Lie $2$-groups}    
  \date{}
\author{{Gr\'egory Ginot,}\cr
{\small{Institut math\'ematique de Jussieu}}\cr
{\small{Universit\'e Pierre et Marie Curie}}\cr
{\small{4, place Jussieu, 75252  Paris, France}}\cr
{\small{e--mail: ginot@math.jussieu.fr}}\cr
{~}\cr
{Ping Xu\thanks{Research partially supported by NSF
grant DMS-0605725 \&  NSA grant H98230-06-1-0047},}\cr
{\small{Department of Mathematics}}\cr
{\small{Penn State University}}\cr
{\small{University Park, PA 16802, U.S.A.}}\cr
{\small{e--mail: ping@math.psu.edu}}\cr
{~}\cr}

\maketitle

\begin{abstract}  In this paper we study the cohomology of (strict) Lie 2-groups.
 We obtain an explicit Bott-Shulman type map in the case of a Lie 2-group corresponding to the crossed module $A\gto 1$.
 The cohomology of the Lie 2-groups corresponding to the universal crossed modules $G\gto \Aut(G)$ and $G\gto \Aut^+(G)$  is the abutment of a spectral sequence involving the cohomology of $GL(n,\Z)$ and $SL(n,\Z)$. When the dimension of the center of $G$ is less than 3, 
we compute explicitly these cohomology groups. 
 We also compute the cohomology of the Lie 2-group corresponding to a crossed module $G\gto[i] H$ for which $\ker(i)$ is compact and $\coker(i)$ is connected, simply connected and compact 
and apply the result to the {\it string} $2$-group.
\end{abstract}    

\section{Introduction}

This paper is devoted to the study of Lie  2-group cohomology.
A Lie 2-group is a Lie groupoid $\Gamma_2\toto \Gamma_1$,
 where both the space of objects $\Gamma_1$ and
the space of   morphisms are Lie groups and
all  the groupoid  structure maps are group morphisms.
This is what is usually refer to as ``groupoids over
groups".  
It is well known that Lie  2-groups
are  equivalent to crossed modules~\cite{BrHi, BaLa}.
By a  {\em crossed module},
we mean   a  Lie group morphism  $G\rto[{i}] H$ 
together with a right action  of $H$ 
on $G$ by automorphisms satisfying certain compatibility conditions.
In this case, $\ker i$ is called the kernel, while
$H/i(G)$ is called cokernel of the crossed module.

Lie 2-groups arise naturally
in various  places in mathematical physics, for
instance, in  higher gauge theory \cite{BaSc}.
They also appeared in the theory  of non-abelian gerbes.
As was shown by Breen \cite{Breen, Br4} (see also \cite{GiSt}), 
a $G$-gerbe is equivalent to
a 2-group  principal bundle in the sense of Dedecker \cite{Dedecker},
where the structure 2-group is the one corresponding to
the crossed module $G\rto[i]\Aut(G)$ with  $i$ denoting  the
 map to the inner automorphisms.

As in the  1-group case, 
associated to  any Lie 2-group $\gm$, there is a simplicial manifold  $N_\com\gm$, called the nerve
of the 2-group. 
Thus one defines the cohomology of a 
Lie 2-group $\gm$ with trivial coefficients  $\rr$ as the
cohomology of this simplicial manifold  $N_\com\gm$ with coefficients $\rr$.
 The latter  can be computed using a double 
 de Rham  chain complex.
 A very natural question arises as to whether there is a Bott-Shulman type map \cite{BS, BSS}
for  such a Lie 2-group. 
Unfortunately, the  answer seems to be out of reach in general. 
However, we are able to describe a class of cocycles
in $\Omega^{3r}([G\rto[i] H])$ generated by elements in
$ S^{}\big((\g^*)^{\g,H}[3]\big)$, the symmetric algebra on the vector space $\big((\g^*)^{\g,H}\big)$ with degree 3. 
Here we denote by $[G\gto[i] H]$  the Lie $2$-group corresponding to the crossed module $G\gto[i] H$.
As a consequence, we explicitly describe, for any abelian group
$A$, cocycles in $\Omega^{\com}([A\gto 1])$ which generate the
cohomology group $H^\com ([A\gto 1])$. 
These cocycles are given by skew-symmetric polynomial functions on the Lie algebra $\a$ of $A$. 
Such an explicit
 map is also obtained in the case  when the
cokernel of $G\gto H$ is finite.
 Our approach is based on the following idea.
 A Lie 2-group $[G\gto H]$ induces a short exact sequence of Lie 2-groups:
$$1\gto{} [\ker i \gto 1]\gto{} [G\gto H] \gto{} [1\gto \coker i]\gto{} 1$$
which in turn induces a fibration of 2-groups.
As a consequence, we obtain  a Leray-Serre spectral sequence.
Discussions
on these topics occupy Sections 4-5.

We also use the spectral sequence to compute the cohomology of a $2$-group $[G \gto[i] H]$ with connected and simply connected compact cokernel $\coker(i)\cong H/i(G))$ and compact kernel $\ker(i)$. 
In general, the cohomology of $[G \gto[i] H]$ depends on a transgression homomorphism $$T:H^3([\ker(i)\gto 1])\gto H^4([1\gto H/i(G)]).$$ 
An example of such $2$-group is given by the string $2$-group~\cite{BCSS} for which we recover computations also idependently due to Baez and Stenvenson~\cite{BaSt}. 

Next we apply our result to study the cohomology of  particular classes
of 2-groups: $[G\gto[i^+] \Aut^+ (G)]$ and
$[G\gto[i] \Aut (G)]$, where $\Aut^+ (G)$ is the
orientation preserving automorphism group of $G$. 
 If $G$ is a semi-simple Lie
group, the result is  immediate since both kernel and
cokernel are finite groups.
However when $G$ is a  general compact
Lie group, the situation becomes
much subtler. 
This is due  to the  fact that  the connected
component of the center $Z(G)$  is a torus $T^n$,
 and therefore $\Out^+ (G)$ and $\Out (G)$ are no longer finite groups.
Indeed they are closely  related to
$SL (n, \zz)$ and $GL(n, \zz)$, 
whose cohomology groups are in general very difficult to compute, and
 still remains an open question for large $n$. 
Nevertheless, we obtain a spectral sequence involving
cohomology of these groups,  converging to the cohomology of the 2-group.
For $n\leq 3$, using a result of Soul\'e \cite{Sou},   we are able to
compute the cohomology groups explicitly.

 One of the  main motivations for studying cohomology of 
 2-groups is to study characteristic classes of gerbes. 
Since $G$-gerbes correspond to principal
 $[G\gto[i] \Aut(G)]$-bundles, 
   any nontrivial cohomology class in $H^\bullet([G\gto[i] \Aut(G)])$ defines a universal characteristic class for $G$-gerbes. 
And  a Bott-Shulman type cocycle allow one to express such a 
   universal characteristic class in terms of geometric data such as connections
   just like in the usual Chern-Weil theory.
  This will be  discussed in detail  in~\cite{GiSt}.

  Note  that the constructions in this paper can be defined in the more general context of weak Lie 2-groups
   as defined by Henriques in~\cite{Hen} since  the cohomology and homotopy groups are defined using the nerve.

\noindent{\bf Acknowledgment:} The authors would like to thank
 A.~Ash, L. Breen, A. Henriques, K. Mackenzie,  C.~Soul\'e,  J. Stasheff
and the referee
 for many useful comments and  suggestions.

\medskip

\noindent{\bf Notations :} Given a (graded) vector space $V$ we denote by $V [k]$ the graded vector space with shifted grading $\left(V[k]\right)^n=V^{n-k}$. Thus if $V$ is concentrated in degree $0$, $V[k]$ is concentrated in degree $k$. 
The graded symmetric (or free commutative) algebra  on a graded vector space $V$ will be denoted by $S^{}(V)$. We write $S^{}(V)^q$ for the subspace of homogeneous elements of total degree $q$, that is,  $S^{}(V)^q=\{x_1\dots x_r \in S^{r}(V) \, / \, r\geq 0 \mbox{ and } |x_1|+\cdots+|x_r|=q \}$.  
 In particular, if $x\in S(V)^{p}$ and $y\in S(V)^{q}$, on has $x\cdot y= (-1)^{pq}y\cdot x$. 
Thus if $V$ is concentrated in even degrees, $S(V)$ is a polynomial algebras.
On the other hand,  if $V$ is  concentrated in odd degrees,
 $S(V)$ is an exterior algebra.

\smallskip

Unless otherwise stated, all cohomology groups are taken with real coefficients.

\section{Crossed modules}\label{sec:Crossed}

 A {\em crossed module}  of Lie groups is a  Lie group morphism  $G\rto[{i}] H$ 
together with a right $H$-action $(h,g)\gto g^{h}$ of $H$ on $G$ by Lie group automorphisms satisfying:
\begin{enumerate}
\item for all $(h,g)\in G\times H$,
  $i(g^{h})=h^{-1}i(g)h$;
\item for all $(x,y)\in G\times  G$, $x^{i(y)}=y^{-1}xy$.
\end{enumerate}
 A (strict) morphism $\big(G_2\rto[i_2] H_2\big) \rightarrow \big(G_1 \rto[i_1] H_1\big)$ of crossed modules is a pair $(\phi:G_2\rto G_1, \psi: H_2\rto H_1)$ of Lie group morphisms  such that $\psi \circ i_2=i_1\circ \phi$ and $\phi(g)^{\psi(h)}=\phi(g^h)$ for all $g\in G_2$, $h\in H_2$.

\smallskip

 There is a well known equivalence of categories between the category of crossed modules and the category of (strict) Lie $2$-groups~\cite{BrHi}.
 Recall that a Lie $2$-group is a group object in the category of Lie groupoids meaning it is a Lie groupoid $\gm_2\toto \gm_1$ where, both, $\gm_2$ and $\gm_1$ are Lie group and all structure maps are Lie group morphisms. 
Such a $2$-group will be denoted by $\gm_2\toto\gm_1\toto \{*\}$. 
The crossed module $G\rto[{i}] H$ gives rise to the $2$-group $G\rtimes H \toto H\toto \{*\}$. 
The groupoid  $G\rtimes H \toto H$ is the transformation groupoid:  the source and target maps $s,t: G\rtimes H \gto H $   are respectively given by $s(g,h)=h$, $t(g,h)= h\cdot i(g)$. 
The (so called vertical) composition is $(g,h)\star (g',h\cdot i(g))=(gg',h)$. 
 The group structure on $H$ is the usual one while the group structure (the so called horizontal composition) on $G\rtimes H$ is the semi-direct product of Lie groups:   $(g,h)*(g',h')=(g^{h'}g',hh')$.   
Conversely, there is a crossed module associated to any Lie $2$-group~\cite{BrHi}.
In the sequel we make no distinctions between crossed modules and $2$-groups.
We use the short notation
$[G\stackrel{i}\longrightarrow H]$ for the Lie $2$-group corresponding to a crossed module $G\gto[i] H$.
 
 \smallskip
 
\begin{definition}\label{D:kernel} Let $(\phi,\psi):\big(G_2\rto[i_2] H_2\big) \rightarrow \big(G_1 \rto[i_1] H_1\big)$ be a morphism of crossed modules with  $\psi$ being a submersion. 
The {\em kernel} of the map $(\phi,\psi)$ is, by definition (see~\cite{Noo}), the crossed module $(G_2\rto[\tilde{i}] H_2\times_{H_1} G_1)$ where $\tilde{i}$ is the natural group morphism induced by $i_2$ and $\phi$. 
The $H_2\times_{H_1} G_1$-action on $G_2$ is induced by the $H_2$-action: $g_2^{(h_2,g_1)}=g_2^{h_2}$. 
 The structure map $H_2\times_{H_1} G_1\gto H_2$ induces a natural crossed module morphism $(G_2\rto[\tilde{i}] H_2\times_{H_1} G_1)\gto \big(G_2\rto[i_2] H_2\big)$.
 \end{definition}
\medskip

A Lie group $G$ can be seen as a Lie 2-group with trivial 2-arrows, {\em i.e.}  as the Lie $2$-group $G\toto G\toto \{*\}$. The associated crossed module is $1\rto  G$. 
It yields an embedding of the category of Lie groups in the category of Lie 2-groups.  
Similar to the case of a group, associated to a Lie $2$-group $\gm: \gm_2\toto \gm_1\toto \{*\}$, there is a simplicial manifold $N_\com \gm$, called its (geometric) nerve. 
 It is the nerve of the underlying  2-category as defined by Street~\cite{Str}.  
In particular, $N_0\gm=\{*\}$, $N_1\gm=\gm_1$ and $N_2\gm$ consists of 2-arrows of $\gm_2$ fitting in a commutative square:
\eq\label{eq:N3}\xymatrix{ & A_1\ar[rd]^{f_0} \ar@{=>}[d]^{\alpha}  &  \\
 A_0\ar[rr]_{f_1} \ar[ru]^{f_2} &&A_2 }
\eneq
$N_2 \gm$ is naturally a submanifold of $\gm_2\times \gm_1\times \gm_1\times \gm_1$. For $p\geq 3$, an element of $N_p\gm$ is a   $p$-simplex (labelled by arrows of $\gm$) such that each subsimplex of dimension $3$ is a commutative 
tetrahedron,
 whose   faces are  given by elements of $N_2\gm$ (see~\eqref{eq:tetrahedron} below or~\cite{Noo, MoSv, Str}).
Also see Remark~\ref{rm:Nerve} below.

\smallskip

 The nerve $N_\com$ defines a functor from the category of Lie $2$-groups to the category of simplicial manifolds. 
The nerve of a Lie group considered as a Lie $2$-group is isomorphic to  the usual (1-)nerve~\cite{Seg}. Taking the fat realization of the nerve defines a functor from Lie $2$-groups to topological spaces.
 In particular, the homotopy groups of a Lie $2$-group can be defined as the homotopy groups of its nerve.

\medskip

 Note that Lie $2$-groups embed evidently in the category of weak Lie $2$-groupoids (see for instance~\cite{BaLa} and~\cite{Hen}).  
There is a notion of fibration for (weak) Lie $2$-groups 
due to Henriques ~\cite[Section 2 and 4]{Hen} (see also 
\cite{Zhu, Zhu2} as well). 
We    also refere the reader to ~\cite{MoSv, Noo} for
an  excellent exposition in the case of discrete $2$-groups.
 In the present paper, however, we only use a special kind of fibrations,
 which is given by the following lemma:  

\begin{lemma}\label{lm:fibration}
Let $(\phi,\psi):\big(G_2\rto[i_2] H_2\big) \rightarrow \big(G_1 \rto[i_1] H_1\big)$ be a morphism of crossed modules with $\phi$ and $\psi$ being surjective submersions. 
Then $(\phi,\psi):[G_2\rto[i_2] H_2] \rightarrow [G_1 \rto[i_1] H_1]$ is a fibration of Lie $2$-groups. 
The kernel of the morphism $(\phi,\psi)$ (as in Definition~\ref{D:kernel}), \emph{i.e.}  the  Lie 2-group  $[G_2\rto[\tilde{i}] H_2\times_{H_1} G_1]$,  is a homotopy fiber of $(\phi,\psi)$ and is equivalent to $[ker(\phi)\rto[i_2] ker(\psi)]$. 
\end{lemma}
\begin{proof} Let $\gm_1$ and $\gm_2$ be the Lie 2-groups corresponding to the crossed modules   $(G_1 \rto[i_1] H_1)$ and $(G_2\rto[i_2] H_2)$ respectively,   and $\Phi:\gm_2\gto \gm_1$ the map induced by $(\phi,\psi):\big(G_2\rto[i_2] H_2\big) \rightarrow \big(G_1 \rto[i_1] H_1\big)$. 
Since $\phi$ and $\psi$ are surjective submersions, $N_m\Phi: N_m\gm_2\gto N_m\gm_1$ is a surjective submersion for all $m$.
 Since $\gm_2$ and $\gm_1$ are (strict) Lie 2-groups, their nerves $N_\com \gm_2$ and $N_\com \gm_1$ are simplicial manifolds  satisfying the Kan condition  for simplicial manifolds as in~\cite[Definiton 1.2 and Definition 1.4]{Hen}. 
 Thus, for all $m,j$, the canonical map $N_m\gm_2=Hom(\Delta^m_\com, N_\com \gm_2) \gto Hom(\Lambda[m,j]_\com, N_\com \gm_2)$ are surjective submersions for $m\leq 2$ and diffeomorphisms for $m>2$.
 Here $\Delta^m_\com$ is the simplicial $m$-simplex and $\Lambda[m,j]_\com$ its $j$th-horn, \emph{i.e.}, the subcomplex generated by all facets containing the $j$th-vertex. The same results holds when $\gm_2$ is replaced by $\gm_1$. 

The   map $N_m\gm_2=Hom(\Delta^m_\com, N_\com \gm_2) \gto Hom(\Lambda[m,j]_\com, N_\com \gm_2)$ and  the map $N_m\gm_2 =Hom(\Delta^m_\com, N_\com \gm_2) \gto Hom(\Delta^m_\com, N_\com \gm_1)=N_m\gm_1$ induced by $\Phi:\gm_2\gto \gm_1$ yields, 
for all $j$, a smooth map from $N_m\gm_2$ to the space $C[m,j]$, which consists of the commutative squares \begin{eqnarray*}  \xymatrix{\Lambda[m,j]_\com  \ar[r] \ar[d] & N_\com \gm_2\ar[d]^{N_\com \Phi}\\ \Delta^m_\com \ar[r] &  N_\com \gm_1\,}\end{eqnarray*} See~\cite[Definition 2.3]{Hen}. 
Note that $C[m,j]$ can be identified with with the fiber product $ Hom(\Lambda[m,j]_\com, N_\com \gm_2)\times_{Hom(\Lambda[m,j]_\com, N_\com \gm_1)} Hom(\Delta^m_\com, N_\com \gm_1)$.
 By the definition of a fibration~\cite[Definition 2.3]{Hen}, it suffices to prove that (for all $m,j$) the map $N_m\gm_2 \gto C[m,j]$ is a surjective submersion .  
For $m>2$, $Hom(\Lambda[m,j]_\com, N_\com \gm_2)\cong N_m\gm_2$. 
Thus $C[m,j]\cong N_m\gm_2$ and we are done.
 For $m=1$, $C[1,j]\cong H_1$ and the map $N_1\gm_2\cong H_2\gto C[1,j]\cong H_1$ is $\psi$. 
For $m=2$, $C[2,j]$ is identified with $H_2^{\times 2} \times G_1$ and the map    $N_2\gm_2\gto C[2,j]$becomes  $H_2^{\times 2} \times G_2 \stackrel{id^{\times 2}\times \phi}\longrightarrow H_2^{\times 2} \times G_1$. 
The latter  is a surjective submersion since $\phi$ is a submersion.

The fiber $F_\com$ of $N_\com \Phi$ is the pullback $pt_\com \times_{N_\com \gm_1} N_\com \gm_2$ where $pt_\com=N_\com[1\gto 1]$ is the point (viewed as a constant simplicial manifold). Thus, $F_\com$ is the nerve of the Lie 2-groups $[ker(\phi) \stackrel{i_2}\gto ker(\psi)]$.
 Here the crossed module structure  of $ker(\phi) \stackrel{i_2}\gto ker(\psi)$ is induced by that of $G_2\rto[i_2] H_2$. 
The inclusions $ker(\phi)\hookrightarrow G_2$ and  the map $ker(\psi) \hookrightarrow H_2 \rto[id\times 1] H_2\times_{H_1} G_1$ yield a crossed module homomorphism $\big(ker(\phi) \stackrel{i_2}\gto ker(\psi)\big) \gto \big( G_2\rto[\tilde{i}] H_2\times_{H_1} G_1\big)$, which is an equivalence of crossed modules. See~\cite{BaLa, MoSv, Noo, Zhu, GiSt} for the definition of equivalence of crossed modules or Lie $2$-groups. 
It follows that $N_\com [G_2\rto[\tilde{i}] H_2\times_{H_1} G_1]$ is weakly homotopic to $F_\com$. Furthemore the following natural diagram
$$\xymatrix{ [G_2\rto[\tilde{i}] H_2\times_{H_1} G_1] \ar[rd] & \\ 
[ker(\phi) \stackrel{i_2}\gto ker(\psi)] \ar[u] \ar[r] & [G_2\rto[i_2] H_2]} $$
 is commutative. 
Thus $[G_2\rto[\tilde{i}] H_2\times_{H_1} G_1]$ is a homotopy fiber of the map $(\phi,\psi):[G_2\rto[i_2] H_2] \rightarrow [G_1 \rto[i_1] H_1]$.
\end{proof}

As far as  this paper is concerned,  it is sufficient
to consider Lemma~\ref{lm:fibration} as
 a definition of a fibration of Lie $2$-groups. In particular all fibrations of Lie 2-groups in this paper arise as in Lemma~\ref{lm:fibration}.
That is, they are   induced by a  morphism $(\phi,\psi)$ of crossed modules 
with both $\phi$ and $\psi$ being  surjective submersions. 

\begin{example} The main examples of interest in this paper are obtained as follows (see Section~\ref{sec:cohomology}).
 Let $G\rto[i] H$ be a crossed module and $\psi: H\gto K$ be a Lie group morphism such that $\psi (i(g))=1$ for all $g\in G$.
 Then the map $(1,\psi):[G\rto[i] H]\rightarrow [1\gto K]$ is a map of 2-groups and it is a fibration if $\psi$ is a surjective submersion. 
The kernel of the map $(1,\psi)$ (as defined in Definition~\ref{D:kernel}) is   the Lie 2-group $[G\rto[i] ker(\psi)]$, which is equal   to the Lie 2-group $[ker(1)\rto[i] ker(\psi)]$.
\end{example}

\begin{remark} Let us recall that a 2-group is a group object in the category of groupoids.
 Then the Lie 2-group $[G_2\rto[\tilde{i}] H_2\times_{H_1} G_1]$ is the (weak) fiber product (of Lie groupoids, see~\cite{MoMr}) $[1\gto 1] \times_{[G_1 \rto[i_1] H_1]} [G_2 \rto[i_2] H_2]$. 
In particular it is the correct fiber product to look at if one is interested in group stacks rather than Lie 2-groups.
\end{remark}

\section{Cohomology of  Lie $2$-groups}\label{sec:cohomology}

The de Rham cohomology groups  of a Lie 2-group $\gm$ are defined as the cohomology groups of the bicomplex $(\Omega^\com(N_\com \gm), \ddR, \dpa)$, where $\ddR:\Omega^p(N_q \gm)\gto\Omega^{p+1}(N_q \gm) $ is de Rham differential and $\dpa: \Omega^p(N_q \gm) \gto
\Omega^p(N_{q+1} \gm)$ is induced by the simplicial structure on $N_\com \gm$ : $\dpa=(-1)^p\sum_{i=0}^{q+1} (-1)^i d_i^*$ where $d_i:N_\com \gm \gto N_{\com-1} \gm$ are the face maps. 
We use the shorter notation $\Otg[\com]$ for the associated total complex. 
Hence $\Otg[n]=\ooplus_{p+q=n} \Og[p](\Ng[q])$ with (total) differential $\ddR+\dpa$.  
We denote by $H^\com(\gm)$ the cohomology  of $\gm$.  
It is well-known that $H^\com(\gm)$ is naturally isomorphic to the cohomology of the fat realization of its nerve $N_\com \gm$ (for instance see~\cite{BSS}). 

\smallskip

The simplicial structure of the nerve $\Ng[\com]$ of a Lie $2$-group $\gm$ gives  rise to a structure of cosimplicial algebra on the space of de Rham forms $\Otg[\com]$. 
Thus, there exists an associative cup-product $\cup: \Otg[\com]\otimes \Otg[\com] \gto \Otg[\com]$  making $(\Otg[\com], \ddR+\dpa, \cup)$ into a differential graded algebra and, therefore, $(H^{\com}(\gm),\cup)$ is a graded commutative algebra. The same holds for singular cohomology.

\smallskip

If $G\rto[i] H$ is a crossed module,  
we denote by $\Ogtot[\com]([G\rto[i] H])$ the total complex of the corresponding Lie $2$-group. A map of Lie 2-groups  $f:\gm\gto G$ induces a simplicial map $N_\com \gm \gto N_\com G$, and  by pullback, a map of cochain complexes
$\Omega^\com(N_\com G)\rto[{f^*}] \Omega^\com(N_\com \gm)$. 
 A similar construction, replacing the de Rham forms by the singular cochains with coefficient in a ring $R$, yields the singular cochain functor $C^\com([G\rto[i] H],R)$ of the Lie $2$-group $[G\rto[i] H]$ whose cohomology $H^\com([G\rto[i] H],R)$ is the singular cohomology with coefficients in $R$. 
If $R=\R$, the singular cohomology groups coincides with the de Rham cohomology groups.  
 The cohomology of a  Lie group considered as a Lie $2$-group is the usual cohomology of its classifying space since, in that case, the $2$-nerve is isomorphic to the $1$-nerve of the Lie group~\cite{Str}.

\medskip
Given a crossed module $G\gto[i] H$ of Lie groups, $i(G)$ is a normal subgroup of $H$. Hence, the projection $H\gto H/i(G)$ induces a Lie $2$-group morphism
\eq
\label{eq:coker} [G\rto[i] H] \longrightarrow [1\gto H/i(G)]
\eneq 
which is a fibration (by Lemma~\ref{lm:fibration}) with the fiber being the 2-group $[G \rto[i] i(G)]$.  The canonical morphism of crossed modules $(\ker(i) \rto 1)\gto (G \rto[i] i(G))$ is an equivalence (see~\cite{BaLa, MoSv, Noo, Zhu, GiSt} for the equivalence of crossed modules or Lie $2$-groups) and in particular, the Lie 2-group $[G \rto[i] i(G)]$ and $[\ker(i) \gto[] 1]$ have weakly homotopic nerves.
It follows that there is a Leray-Serre spectral sequence.
\begin{lemma}\label{lm:Leray}
There is a converging spectral sequence of algebras
\eq \label{eq:Leray} L_2^{p,q}=H^p([1\gto H/i(G)], \mathcal{H}^q([\ker(i) \rto 1])) \Longrightarrow  H^{p+q}([G\rto[i] H]) \eneq
where $\mathcal{H}^q([\ker(i) \rto 1])$ is the de Rham cohomology viewed as a local coefficient system on $[1\gto H/i(G)]$.
\end{lemma}
\begin{proof} It follows from the main theorem of~\cite{And} that the realization of the map  $[G\rto[i] H] \rto{}  [1\gto H/i(G)]$ is a quasi-fibration  with homotopy fiber being the (fat) realization of $[\ker(i)\rto 1]$ (since this fat realization is homotopic to the fat realization of $[G \rto[i] i(G)]$). In fact, one can show that this quasi-fibration is indeed a fibration.  
The spectral sequence~\eqref{eq:Leray}  is the Leray-Serre spectral sequence of this (quasi)fibration. 
\end{proof}

By the  same argument,
 it also  follows that there is a long exact sequence of homotopy groups 
\footnote{The sequence should not be confused with the long exact sequence of simplicial homotopy groups in~\cite[Section 3]{Hen}.}
{\small \eq \label{eq:homotopygroups}
\dots \pi_1([1\!\gto H/i(G)])\gto \pi_0([\ker(i)\! \rto \! 1])\gto \pi_0([G\!\rto[i]\! H]) \gto\pi_0([1\!\gto\! H/i(G)])\gto 0.
\eneq}

\begin{remark}
The algebra structures in Lemma~\ref{lm:Leray} are induced by the algebra structure on the  singular   or de Rham cohomology of the respective Lie $2$-groups. 
\end{remark}

\begin{remark}\label{rm:Leray}
A similar proof implies that if $[G_2\rto[i_2] H_2] \gto{} [G_1\rto[i_1] H_1]$ is a fibration of $2$-groups with fiber $F$, then there is a Leray spectral sequence $$L_2^{p,q}=H^p([G_1 \rto[i_1] H_1], \mathcal{H}^q(F)) \Longrightarrow  H^{p+q}([G_2\rto[i_2] H_2]).$$
\end{remark}

\begin{remark}
In the special case of discrete $2$-groups, the Leray-Serre spectral sequence~\eqref{eq:Leray} has been studied in~\cite{DaPi}. In this rather different context, the higher differentials in the spectral sequence are related to the $k$-invariant of the crossed module.  
\end{remark}

\smallskip

We now give a more explicit description of the complex $\OtGH[\com]$ in degree $\leq 4$ which will be needed  in Sections~\ref{sec:abelian} and~\ref{sec:GAUTG}. 
Until the end of this Section, we denote by $\gm$   the $2$-group $G\rtimes H\toto H\toto\{*\}$ associated to the crossed module $G\rto[i] H$.
 One has   $\Ng[0]=*$ and $\Ng[1]=H$. Since there is only one object in the underlying category, all 1-arrows can be composed. Thus, a triangle as in equation~\eqref{eq:N3} is given by a 2-arrow $\alpha \in G\rtimes H$ and  a $1$-arrow $f_0$. Hence, $\Ng[2] \cong (G\times H)\times H$. 
With this choice, for $(g,h,f)\in \Ng[2]$, the corresponding 2-arrow $\alpha$ and 1-arrows $f_0,f_1,f_2$ in equation~\eqref{eq:N3} are respectively given by 
\eq
\label{eq:N32group} \alpha= (g,h),\qquad   f_0=f, \qquad  f_1=h\cdot i(g), \qquad f_2=h\cdot f^{-1}.
\eneq
The three face maps $d_i:\Ng[2]\gto \Ng[1]$ ($i=0,1,2$)  are given by $d_i(g,h,f)= f_i$, $i=0,1,2$ (see equation~\eqref{eq:N32group}). 
 \begin{remark}Of course the choice of $f_0$ is a convention, we could have equivalently chosen to work with $f_2$.  
\end{remark}
$\Ng[3]$ is the space of commutative tetrahedron labelled by objects and arrows of $\Gamma$
\eq \label{eq:tetrahedron}&&
{\xymatrix@C=3pt@R=12pt@M=6pt{ &&& A_3   &&& \\
                               &&&&  && \\
                            &&& A_1 \ar[uu]^{f_{02}} \ar@{=>}[ul]^{\alpha_2}
                                \ar@{=>}[d]^{\alpha_3} \ar[drrr]^{f_{03}} &     \ar@{=>}[ul]_{\alpha_0}  &&   \\
                        A_0 \ar[rrrrrr]_{f_{13}} \ar[urrr]^(0.65){f_{23}}   \ar[rrruuu]^{f_{12}}
       && \ar@{:>}[ul]_<{\alpha_{1}} &&&& A_2 \ar[llluuu]_{f_{01}}   }}
\eneq
The commutativity means
that  one has $(\alpha_3 * f_{01})\star \alpha_1 = (f_{23}*\alpha_0)\star \alpha_2$, where $\star$ is the vertical multiplication of 2-arrows and $*$ is the horizontal multiplication. Since there is only one object, such a tetrahedron is given by $\alpha_0,f_{01},\alpha_2$ and $\alpha_3$ satisfying $s(\alpha_3)=s(\alpha_2).(t(\alpha_0))^{-1}.s(\alpha_0).f_{01}^{-1}$. Thus $\Ng[3]\cong G^{3}\times H^3$. The face maps $d_i$ ($i=0...4$)  are given by the restrictions to the triangle which doesn't contain $A_i$ as a vertex. Thus,  given $(g_0,g_2,g_3,h_0,f_{01},h_2)\in G^3\times H^3$, one has 
\eq  d_0(g_0,g_2,g_3,h_0,f_{01},h_2) &=& (g_0,h_0,f_{01})\label{eq:d0}\\
   d_1(g_0,g_2,g_3,h_0,f_{01},h_2)&=& ((g_3^{-1})^{f_{01}}g_0^{h_2i(g_0^{-1})h_0^{-1}}g_2,h_2,f_{01})\label{eq:d1}\\
 d_2(g_0,g_2,g_3,h_0,f_{01},h_2)&=& (g_2,h_2,h_0.i(g_0))\label{eq:d2} \\ \label{eq:d3}
d_3(g_0,g_2,g_3,h_0,f_{01},h_2)&=& (g_3, h_2.i(g_0^{-1})f_{01}^{-1},h_0.f_{01}^{-1}).
\eneq
\begin{remark}
 The choice of indices in $(g_0,g_2,g_3,h_0,f_{01},h_2)$ is reminiscent of the tetrahedron~\eqref{eq:tetrahedron}. That is the two arrow $\alpha_0 =(g_0,h_0)\in G\times H$, the $1$-arrow from $A_2$ to $A_3$ is $f_{01}$ and so on \dots For instance, the $2$-arrow $\alpha_1=(g_1,h_1)\in G\times H$ is given by Equation~\eqref{eq:d1}, i.e, $g_1= (g_3^{-1})^{f_{01}}g_0^{h_2i(g_0^{-1})h_0^{-1}}g_2$ and $h_1=h_2$. 
\end{remark}

Applying the differential form functor, we get
\anum
\item $\OtGH[0]=\Og[0](*)\cong \R$,
\item $\OtGH[1] = \Og[0](H)$,
\item $\OtGH[2]=\Og[1](H)\oplus \Og[0](G\times H\times H)$. The differentials from $\OtGH[1]$ to $\OtGH[2]$ are given by $\ddR:  \Og[0](H)\gto \Og[1](H)\subset \OtGH[2]$ and $\dpa =d_0^*-d_1^*+d_2^*:\Og[0](H)\gto \Og[0](G\times H\times H)\subset \OtGH[2]$. 
\item $\OtGH[3]=\Og[2](H) \oplus \Og[1](G\times H\times H) \oplus \Og[0](G^3\times H^3)$. The differential are similar to the previous ones.
\item $\OtGH[4]=\Og[3](H) \oplus \Og[2](G\times H\times H) \oplus \Og[1](G^3\times H^3)\oplus \Og[0](\Ng[4])$. 
\enum
\begin{remark} \label{rm:Nerve}
For $p\geq 4$, an element in $\Ng[p]$ is a commutative $p$-simplex labelled by arrows of $\gm$ whose faces of dimension $2$ are elements of $\Ng[2]$ with compatible edges. 
Denoting $A_0,\dots, A_p$ the vertices of the $p$-simplex,  the commutativity implies that it is enough to know all the $2$-faces containing $A_0$. Reasoning as for $\Ng[3]$, it  follows that $\Ng[p]\cong G^{\frac{p(p-1)}{2}}\times H^p$. Details are left to the reader.  
\end{remark}

Let $\g$ be the Lie algebra of $G$.    There is an obvious map $ (\g^*)^{\g}\hookrightarrow \Og[1](G)$ which sends $\xi\in \g^*$ to its left invariant $1$-form $\xi^L$. By composition we have a map
\eq \label{eq:I}
  (\g^*)^{\g}\hookrightarrow \Og[1](G) \rto[p_1^*]\Og[1](G\times H\times H) \hookrightarrow \OtGH[3]
\eneq 
where $p_1:G\times H\times H\gto G$ is the projection.

 The action of $H$ on $G$ induces an action of $H$ on $\g$, and therefore  an action on $\g^*$. 
The map I above clearly restricts to $(\g^*)^{\g, H}$, the subspace of $\g^*$ consisting of elements  both $\g$ and $H$ invariant. Assigning the degree 3 to elements of $(\g^*)^{\g, H}$, {\it i.e.}, replacing  $(\g^*)^{\g, H}$ by  $(\g^*)^{\g, H}[3]$,  we have the following 

\begin{proposition}\label{pro:I}
The map $I:  \big((\g^*)^{\g,H}[3],0\big)\gto \big(\OtGH[\com], \ddR +\dpa\big)$ is a map of cochain complexes, {\em i.e.}, $(\ddR+\dpa)(I)=0$.
\end{proposition}
\begin{proof}
We have $(\g^*)^{\g}\cong \zg^*$, where $\zg$ is the center of the Lie algebra $\g$. 
Since the de Rham differential vanishes on $(\g^*)^{\g}$, it remains to prove  that $\dpa \circ I=0$.
 For any $\xi \in (\g^*)^\g$ and left invariant vector fields $\overleftarrow{X}$, $\overleftarrow{Y}$,
  \eqn m^*(\xi^L)(\overleftarrow{X}_g,\overleftarrow{Y}_h)=\xi^L(m_*(\overleftarrow{X}_g,\overleftarrow{Y}_h))&=&\xi^L(\overleftarrow{X}_{gh},\overleftarrow{{\rm Ad}_g Y}_{gh})\\
&=& \xi(X) +\xi(Y)\\
&=&\big(p_1^*(\xi^L)+p_2^*(\xi^L)\big)(\overleftarrow{X}_g,\overleftarrow{Y}_h),\eneqn where $m,p_1,p_2: G\times G\gto G$ are respectively the product map and the two projections. 
If, moreover, $\xi\in (\g^*)^{\g,H}$, then  $\tilde{m}^*=\tilde{p_2}^*$, where $\tilde{m}, \tilde{p_2}:G\times H\gto G$ are respectively the action map and the projection. 
Since $I(\xi) \in \Omega^1(G)\subset \Omega^1(G\times H\times H)\subset \OtGH[3]$, the result follows from a simple computation using formulas~\eqref{eq:d0}-\eqref{eq:d3}. 
\end{proof}

\smallskip

By Proposition~\ref{pro:I}, the images of the map $I:(\g^*)^{\g,H}[3]\gto \OtGH[3]$ are automatically  cocycles. Recall that $S^{}\big((\g^*)^{\g,H}[3]\big)$ is the free graded commutative algebra on the vector space $(\g^*)^{\g,H}[3]$ which is concentrated in degree 3. 
Thus $S^{}\big((\g^*)^{\g,H}[3]\big)$ is
indeed  an exterior algebra.  
By the universal property of   free graded commutative algebras, we obtain : 
\begin{corollary}\label{cor:I}
The map $I: (\g^*)^{\g, H}[3]\gto H^{3}([G\rto[i] H])$ extends uniquely to a morphism of graded commutative algebras  
$$I:  S^{}\big((\g^*)^{\g,H}[3]\big)^\com \gto H^{\com}([G\rto[i] H]) .$$  In fact, the class $I(\xi_1\cdot \dots \cdot \xi_r)$, $\xi_1,\dots, \xi_r\in (g^*)^\g$, is represented by the cocycle $I(\xi_1)\cup\dots \cup I(\xi_r)\in \Omega^{3r}([G\rto[i] H])$. 
\end{corollary}

\section{Cohomology of $[A\rto 1 ]$}\label{sec:abelian}
The following lemma is  well-known. 
\begin{lemma}\label{lm:KZ3}
The nerve $N_\com([S^1\rto 1])$ is a $K(\Z,3)$-space. 
\end{lemma}
\begin{proof}
Since $\Z$ is discrete, $ [\Z \rto 1]$ is a $K(\Z,2)$-space (see for instance~\cite{Lod, MoSv, Noo}). Furthermore, $[\R\rto 1]$ is homotopy equivalent to $[1\rto 1]$. Thus
the result follows from the fibration of 2-groups $
 [\Z \rto 1] \longrightarrow [\R\rto 1] \longrightarrow [S^1\gto 1]$.
\end{proof}
Let $A$ be an abelian compact Lie group with Lie algebra $\a$. 
Then $[A\gto 1]$ is a crossed module. 
By Corollary~\ref{cor:I}, we have a map  $I:S^{}(\a^*[3])\gto H^{\com}([A\gto 1])$.
\begin{proposition}\label{lm:A} Let $A$ be an abelian compact Lie group with Lie algebra $\a$.
The map $I:S^{}(\a^*[3])^\com \gto H^{\com}([A\gto 1])$ is an isomorphism of graded algebras.
\end{proposition}
\begin{proof}
Since our cohomology groups have real coefficients,  it is sufficient
 to consider the case where $A$ is a torus $T^{k}$. 
Indeed, writing $A_0\cong T^k$ for the connected component of the identity in $A$, we have a fibration: $$[A_0\rto 1] \gto[] [A\rto 1] \gto[] [A/A_0 \rto 1].$$ Since $A$ is compact, $A/A_0$ is a finite group.  Thus $N_\com[A/A_0\rto 1]$ is a $K(A/A_0,2)$-space and in particular is simply connected. 
Then the Leray spectral sequence (Lemma~\ref{lm:Leray} and Remark~\ref{rm:Leray}) simplifies as $$L_2^{\com, \com}= H^\com(K(A/A_0,2))\otimes H^\com([A_0\gto 1])\Longrightarrow H^\com([A\gto 1]).$$ Since  $A/A_0$ is finite, $H^{i>0}(K(A/A_0,2))\cong 0$. Hence $H^\com([A\gto 1])\stackrel{\sim}\gto H^\com([A_0\gto 1])$. 

\smallskip

Now, assume $A=T^k$. The K\"unneth formula implies that $H^{\com}([A\gto 1])\cong \big(H^\com([S^1\rto 1])\big)^{\otimes k}$ as an algebra. Since $I$ is a morphism of algebras,  it is  sufficient  to consider the case  $k=1$, {i.e.}, $A=S^1$. 

\smallskip

Lemma~\ref{lm:KZ3} implies that $H^{\com}([S^1\gto 1])\cong S(x)$, where $x$ is of degree 3.  
It remains to prove  that the map~\eqref{eq:I} $$I:\R \gto \Omega^{1}(S^1)\gto \Omega^3([S^1\gto 1])$$ generates the degree 3 cohomology of $[S^1\gto 1]$, {\it i.e.}, that $I(1)$ is not a coboundary in $\Omega^3([S^1\gto 1])$. 
Clearly $I(1)$ is the image of the fundamental $1$-form on $S^1$ by the inclusion $\Omega^1(S^1) \hookrightarrow \Omega^3([S^1\gto 1])$.
 By Section~\ref{sec:cohomology}, it is obvious that $\Omega^2([S^1\gto 1])\cong \Omega^0(S^1)$ and that the only component of the coboundary operator $\delta: \Omega^2([S^1\gto 1]) \gto \Omega^3([S^1\gto 1])$ lying in $\Omega^1(S^1)\hookrightarrow \Omega^3([S^1\gto 1])$ is the de Rham differential $\ddR: \Omega^0(S^1)\gto \Omega^1(S^1)$. Since the fundamental $1$-form is not exact, the result follows.
\end{proof}

\section{The case of a finite cokernel}\label{sec:finite}
In this section, we consider the particular case of a Lie $2$-group $[G\gto[i] H ]$ with finite cokernel.
\begin{theorem}\label{th:finitecokernel}
Let $[G\gto[i] H]$ be a Lie 2-group with finite cokernel $C:=H/i(G)$ and compact kernel $\ker(i)$. Let $\k$ be the Lie algebra of $\ker(i)$. There is an  isomorphism of graded algebras
\eqn
H^\com([G\gto[i] H]) \cong \left(S^{}\big(\k^*[3]\big)^{\com}\right)^C.
\eneqn
In particular, the cohomology is concentrated in degree $3q$, $q\geq 0$.
\end{theorem} 
\begin{proof}
The $2$-group $[1\gto C]$ is naturally identified with the $1$-group $C$. 
Thus its nerve $N_\com [1\gto C]$ coincides with the classifying space $BC$ of $C$.
 Furthermore, since $C$ is finite (thus discrete),  the cohomology  (with local coefficients) $H^\com([1\gto C], \mathcal{H}^q([\ker(i) \rto 1]))$ is isomorphic to the  usual group cohomology $H^\com(C, H^q([\ker(i) \rto 1])$,  where the $C$-module structure on $H^q([\ker(i) \rto 1]))$ is induced by the $C$-action  on $\ker(i)$. 
Since $H^q([\ker(i) \rto 1])$ is an $\R$-module and $C$ is finite, the cohomology $H^\com(C, H^q([\ker(i) \rto 1]))$ is concentrated in degree zero so that the spectral sequence of Lemma~\ref{lm:Leray} collapses. Hence 
{\small \eqn
H^\com([G\gto[i] H]) \cong H^0(C, H^q([\ker(i) \rto 1]))\cong  H^q([\ker(i) \rto 1]))^C \cong S\left(^{}\big((\k^*)^{\k}[3]\big)^q\right)^C
\eneqn}
According to Proposition~\ref{lm:A}, they are also isomorphic as algebras due to  the multiplicativity of the spectral sequence and the freeness of $S^{}\big((\k^*)^{\k}[3]\big)$.
\end{proof}

\begin{remark}   
One can find explicit generators for the cohomology $H^\com([G\gto H])$ as follows. For all $y\in \ker(i)$, $x\in G$, $y^{-1}xy=x^{i(y)}=x$. 
Thus $K\subset Z(G)$ and $\zg$ splits as a direct sum $\zg\cong \k \oplus \mathfrak{n}$. We denote by $J$ the map $\g^\g\cong \zg\gto \k$.  
The composition of $J^*: \k^*\gto \g^*$ with the map~\eqref{eq:I} is the map  $$\tilde{I}:\k^* \gto[J^*] \left({\g^*}\right)^\g\gto \Omega^1(G\times H^2)\subset \Omega^3_{\rm tot}([G\gto H]).$$ If $x_1.\dots x_q\in S^{q\geq 2}\big(\k^*[3]\big)$, then $I(x_1)\cup \dots \cup I(x_q)$ lies in $\Omega^q(N_{2q}([G\gto H]))\subset \Omega^{3q}_{\rm tot}([G\gto H])$. 
Note that the action of $h\in H$ on $K$ depends only on the class of $h$ in $C$. Since   $C$ is finite, it follows that, for any $x\in S^q(\k^*[3])$,  $\tilde{I}(x)$ is a cocycle if and only if  $x$ is $C$-invariant.  Let  $\check{I}(x)=\tilde{I}\big(\sum_{c\in C}  x^c\big)$.  
Then   $\check{I}(x)$ is indeed a cocycle and $\check{I}\big({\k^*}\big)$ generates the cohomology  $H^\com([G\gto[i] H])$.  
\end{remark}

\bigskip

 Let $1\gto A\gto G\gto[p] H\gto 1$ be a Lie group central extension.
 Since $A$ is central, there is a canonical action of $H$ on $G$.
 It is easy to see that $G\gto[p] H$ is a crossed module.
\begin{corollary}
Let $[G\gto[p] H]$ be the Lie $2$-group corresponding to a central extension of $H$ by a compact abelian group $A$. 
There is an  isomorphism of graded algebras
$$H^\com([G\gto[p] H]) \cong S^{}(\a^*[3])^\com$$ where $\a$ is the Lie algebra of $A$.
\end{corollary}
Recall that  $S^{}(\a^*[3])^\com$ is a graded commutative algebra generated by degree 3 generators (given by any basis of $\a^*$).  
\begin{proof}
Since $G\gto[p] H$ is a surjective submersion, the cokernel $H/p(G)\cong \{*\}$ is trivial. Moreover the kernel of $[G\gto[p] H]$ is $[A\gto 1]$. 
Hence the conclusion follows from Theorem~\ref{th:finitecokernel}. 
\end{proof}
\begin{remark}\label{rem:centralextension}
Identifying the crossed module $A\gto 1$ with the kernel of $G\gto[p] H$ yields a canonical morphism of $2$-groups $\rho:[A\gto 1]\gto[] [G\gto H]$. It follows from the proof of Theorem~\ref{th:finitecokernel} that the isomorphism $H^\com([G\gto[p] H]) \stackrel{\sim}\gto  S^{}(\a[3])^\com$ is given by the composition
$$S^{}(\a[3])^\com\stackrel{\sim} \longrightarrow H^\com([A\gto 1]) \mathop{\longrightarrow}\limits^{\rho^*}_{\sim} H^\com([G\gto[p] H]). $$
\end{remark}

\begin{example}
Let $G$ be a compact Lie group. It is isomorphic to a quotient of $Z\times G'$ by a central finite subgroup. Here $G'$ is the commutator subgroup of $G$.
 Hence there is a map $G\gto \Aut(G')$ yielding a Lie 2-group $[G\gto \Aut(G')]$ through the action of $\Aut(G')$ on $G'$ (see Section~\ref{sec:GAUTG} below). 
Theorem~\ref{th:finitecokernel} implies that
$$H^{\com}([G\gto \Aut(G')]) \cong S^{}\big((\g^*)^{\g}[3]\big)^\com.$$  
\end{example}

\section{The case of a connected compact cokernel}

The results of Section~\ref{sec:cohomology} can
be applied to a more general type   of 2-groups $[G\gto[i] H]$,
 where $G$ and  $H$ are  Fr\'echet Lie groups (thus possibly infinite dimensional).
 See~\cite{BCSS} for more details on Fr\'echet Lie 2-groups.
In such a case, instead of de Rham cohomology, singular cohomology with real coefficients
can be used.

We start with the following lemma.

\begin{lemma}\label{lm:LerayCompactCokernel}
Let $G$ and  $H$ be   Fr\'echet Lie groups.
Assume that $C=H/i(G)$ is a  connected compact Lie group,
 and  $\ker(i)$ is  compact. 
Then the third page $L_3^{\com,\com}$ of the Leray spectral sequence~\eqref{eq:Leray} is concentrated in bidegree $(p,3q)$ ($p,q\geqslant 0$),
 and 
\eq \label{eq:LerayCompactCokernel} L_4^{p,3q}=H^p(BC) \otimes  S^{q}(\a^*[3]).  \eneq
Here $BC$ is the classifying space of $C=H/i(G)$,
 and $\a$ is  the Lie algebra of $A=\ker(i)$.  
\end{lemma}
Note that, since $S(\a^*[3])$ is a graded commutative algebra\footnote{$S(\a^*[3])$ 
is in fact an exterior algebra, since it is generated by odd degree generators.}
generated by elements of degree 3,
  it  lies in degree $3q$ (where $0\leq q\leq \dim(\a)$).

\begin{proof}
 Note that $C$ is the  cokernel $[1\gto[] C]$ of $[G\gto[i] H]$ 
(see Section~\ref{sec:Crossed}). Since $C=H/i(G)$ is  connected,
its classfying space  $BC$ is simply connected.  
It follows that the $L_2^{i, j}$ term of the Leray spectral sequence 
in Lemma~\ref{lm:Leray} is isomorphic to
$$L_2^{i,j}\cong H^i(BC) \otimes H^j([A\gto[] 1] )$$ as an algebra.
By Proposition~\ref{lm:A}, $ H^\com([A\gto[] 1] )\cong S(\a^*[3])^\com$ is concentrated in degree $3q$ ($q\geq 0$).
 Since the differential $d_2:L_2^{i,j}\gto L_2^{i+2,j-1}$ is a derivation, it follows that $d_2=0$ for degree reason. Similarly, $d_3=0$. 
Thus $L_4^{\com,\com}\cong L_3^{\com,\com}\cong L_2^{\com,\com}$.  
\end{proof}

The (higher) differential $d_4:L_4^{i,j}\gto L_4^{i+4,j-3}$ induces 
a transgression homomorphism \eq \label{eq:Tr} 
T_{}: \a^*\cong L_4^{0,3} \gto[d_4] L_4^{4,0}\cong H^4(BC).
\eneq 

\begin{proposition}\label{prop:Tr} 
Under the same hypothesis as in Lemma \ref{lm:LerayCompactCokernel}, there is a natural linear isomorphism $$H^\com([G\gto[i] H]) \cong \left(H^\com\big(BC\big){/(\im(T_{})}\right)
\otimes S^{}\big(\ker(T)[3]\big)^\com $$ which is an algebra isomorphism if we assume
furthermore that  $C=H/i(G)$  is  simply connected. 

 Here $BC$ is the classifying space of $C$ and $\im T$ is
 the ideal generated by the image of $T$.
\end{proposition}
\begin{proof}
 Since $d_4:L_4^{i,j}\gto L_4^{i+4,j-3}$ is a derivation, it is 
uniquely determined by $T$. 
From Lemma~\ref{lm:LerayCompactCokernel},
it  follows that
 $$L_5^{\com,\com}\cong \left(H^\com\big(BC\big){/(\im T)}\right)\otimes 
S^{}\big(\ker(T)[3]\big)^\com.$$ 
For degree reasons, $d_{r}=0$ for all $r\geqslant 5$. 
Thus $L_5^{\com,\com}\cong L_\infty^{\com,\com}$ as an algebra and the linear isomorphism $H^\com([G\gto[i] H]) \cong \left(H^\com\big(BC\big){/(\im(T_{})}\right)
\otimes S^{}\big(\ker(T)[3]\big)^\com $ follows since our ground ring is a field.
If $C$ is further simply connected, then $H^\com(BC)$ is a
 polynomial algebra with  generators
 $x_i$ of even degree $|x_i|=2i$, $i\geqslant 2$. 
In particular, $H^4(BC)$ has no decomposable elements, 
thus $L_\infty^{\com,\com}$ is a polynomial algebras with 
graded generators.
 It follows that $L_\infty^{\com,\com}\cong H^\com([G\gto[i] H]) $ as an algebra.
\end{proof}

 As an application,
below  we compute  the cohomology of \emph{string 2-group}~\cite{BCSS} $String(G)$.
 Let $G$ be   a connected and simply connected compact simple Lie group.
There is a unique left invariant closed $3$-form $\nu$ on $G$,
 which generates $H^3(G,\mathbb{Z})\cong \mathbb{Z}$. By transgression,
 the form $\nu$ corresponds to  a class $[\omega]\in H^4(BG, \mathbb{Z})$,
 which determines the basic central extension~\cite{PS, BCSS} $$1\gto S^1 \gto \widetilde{\Omega G}\gto[\tilde{p}] \Omega G\gto 1$$ of the  based (at identity) loop group  $\Omega G$ of $G$. 
Associated to $\nu$ is a (homotopy class of) map $\Omega B G \gto G \gto K(\Z,3)\cong [S^1\gto 1]$ which induces an isomorphism on $\pi_3$. Let $PG$ denote
 the space of paths $f:[0,1]\gto G$ starting at the identity. 
The conjugation action of $PG$ on $\Omega G$  lifts to $ \widetilde{\Omega G}$. 
The string 2-group (see~\cite{BCSS}) is the Fr\'echet $2$-group corresponding to the crossed module 
$$String(G):=[\widetilde{\Omega G} \gto[p] PG],$$
 where $p$  is the composition $$p: \widetilde{\Omega G}\gto[\tilde{p}] \Omega G \hookrightarrow  PG.$$ 
By construction, $\ker(p)\cong S^1$,  $PG/p(\widetilde{\Omega G}) \cong G$  and, also $\pi_3(String(G))=0$ (as follows from~\cite[Theorem 3]{BCSS}).  
Recall that the cohomology $H^\com(G)$ is the exterior algebra on generators $x_1,\dots, x_r$,
 where $x_i$ is of degree $2e_i+1$ and $e_1,\dots, e_r$ 
are the exponents of $G$. 
Note that we can choose $x_1=\nu$. Similarly $H^\com(BG)$ is the polynomial algebra on generators $y_1,\dots,y_r$ of degree $|y_i|=2e_i$,
 where $y_1$ can be taken to be $[\omega]$. 
To apply Proposition~\ref{prop:Tr}, it suffices to compute the transgression
homomorphism $T:\mathbb{R}\gto H^4(BG)\cong \mathbb{R}$, where the domain
 $\mathbb{R}$   is identified with the Lie algebra of $S^1$.
 Since $[\omega]\in H^4(BG)$ is obtained by  the
transgression from $[\nu]\in H^3(G)\cong H^3(N\Omega G)$, it follows
 that $T(1)$ is the generator of $H^4(BG)$. 
Indeed, there is a  commutative diagram of Fr\'echet 2-groups fibrations
$$\xymatrix{&[1\gto \Omega G] \ar[r] \ar[d]_{j} &[1\gto PG] \ar[r]^{ev} \ar[d] & [1\gto G] \ar@{=}[d] \\ [S^1\gto 1] \ar[r] & [\widetilde{\Omega G}\gto[\tilde{p}] \Omega G ] \ar[r]  & String(G) \ar[r]^{ev}  & [1\gto G] } $$ where the right horizontal arrows are induced by $ev: PG\gto[]  G$, $f\mapsto f(1)$ and the canonical inclusion $[S^1\gto[] 1] \gto[] [\widetilde{\Omega G}\gto[\tilde{p}] \Omega G ]=\ker(ev)$ is an equivalence of Fr\'echet $2$-groups. 
Thus the transgression map $T$ is the composition $T'\circ j^*$ where $T':H^3(G)\cong H^3([1\gto \Omega G])\gto H^4([1\gto G])\cong H^4(BG)$ is the transgression map associated to the fibration $[1\gto PG]\gto[ev] [1\gto G]$. 
Since $PG$ is contractible, $T'(\nu)$ is a generator\footnote{as for the case of the ''universal'' fibration $G\gto EG\gto BG$} of $H^4(BG)$. 
It also follows from the exact sequence~\eqref{eq:homotopygroups} that $j$ is an isomorphism on $\pi_3$ and so is $$j^*:\R\cong H^3([S^1\gto 1])\gto H^3([1\gto \Omega G])\cong H^3(B\Omega G)\cong H^3(G).$$ 
Hence $T=T'\circ j*:\R\gto H^4(BG)\cong \R$ is an isomorphism.
 Thus, we recover the following result of  Baez-Stevenson~\cite{BaSt}:

\begin{proposition}
$$H^\com(String(G))\cong S(y_2,\dots,y_r) \cong H^\com(BG)/([\omega]), $$
 where the $y_i$s are the generators of $H^4(BG)$.
\end{proposition}

\section{The case of $[G\gto \Aut^+(G)]$ and $[G\gto \Aut(G)]$}\label{sec:GAUTG}

Let $G$ be a compact Lie group. 
There is a canonical morphism $G\gto[i] \Aut(G)$ given by  inner automorphisms which is also a crossed module.  
Since inner automorphisms are orientation preserving, we also have a crossed module $G\gto[i^+] \Aut^+(G)$ where $\Aut^+(G)$ is the group of orientation preserving automorphisms.

\medskip

Now, assume $G$ is a semi-simple Lie group.  Then both $\Out(G)$ and  $\Out^+(G)$ are finite groups.
 Moreover, $\ker(i)$ and $\ker(i^+)$ are finite too. Thus, by Theorem~\ref{th:finitecokernel}, we obtain
\begin{proposition}\label{ex:ssimple} Let $G$ be a semi-simple Lie group.
$$H^{n}([G\gto \Aut(G)]) \cong H^{n}([G\gto \Aut^+(G)]) \cong \left\{ \begin{array}{l} 0 \quad \mbox{ if } n>0, \\
\R \quad \mbox{ if } n=0.\end{array} \right.$$  
\end{proposition}

For general compact Lie groups, the cohomology of $[G\gto[i^+] \Aut^+(G)]$ and $[G\gto[i] \Aut^+(G)]$ can be computed with the help of spectral sequences. 
\begin{theorem}\label{th:main}
If $G$ is a  compact Lie group, there are  converging spectral sequences of graded commutative algebras
\eq \label{eq:E+}
{E_2^+}^{p,q} &=& H^p\big(SL(n,\Z), S^{}((\g^*)^\g[3])^q\big) \Longrightarrow H^{p+q}\big([G\gto[i^+] \Aut^+(G)])\\
 \label{eq:E} E_2^{p,q} &=& H^p\big(GL(n,\Z), S^{}((\g^*)^\g[3])^q\big) \Longrightarrow H^{p+q}\big([G\gto[i] \Aut(G)])
\eneq
where $n=\dim ({(\g^*)}^\g)$ is the dimension of $({\g^*})^\g$, 
and  the $SL(n,\Z)$-action (or $GL(n,\Z)$-action) on $S^{}((\g^*)^\g[3])^q$ is induced by the natural action on $({\g^*})^\g\cong \R^n$.

\smallskip

In particular the spectral sequences are concentrated in bidegrees $(p, 3k)$ ($p$ and $k\geq 0$) and 
\eq \label{eq:E0+}
{E_2^+}^{0,q\neq 0, 3n} \cong 0 & \mbox{ and } & {E_2^+}^{0,0} \cong {E_2^+}^{0,3n}\cong \R\\
\label{eq:E0} E_2^{0,q>0}\cong 0 &\mbox{ and } & E_2^{0,0}\cong \R
\eneq
\end{theorem}

\begin{proof}
Let $\g$ be the Lie algebra of $G$ and $\zg$ the Lie algebra of its center $Z(G)$.
Then $\zg^*=(\g^*)^{\g}$.
Since the kernel of $G\gto[i] \Aut(G)$ is  $Z(G)$,  we have the fibration
\eq
\label{eq:fibration} [Z(G)\rto 1] \stackrel{j}\longrightarrow [G\rto[i] \Aut(G)]  \longrightarrow [1\rto \Out(G)], 
\eneq
where $j$ is the inclusion map. 
 By Lemma~\ref{lm:Leray}, we have a spectral sequence $H^p([1\gto \Out(G)], \mathcal{H}^q([Z(G) \rto 1])) \Longrightarrow  H^{p+q}([G\rto[i] \Aut(G)])$ and similarly for $[G\gto[i^+] \Aut^+(G)]$. 
According to  Proposition~\ref{lm:A}, $H^q([Z(G)\gto 1])\cong S^{}((\g^*)^\g[3])^q$. 
Since $G$ is compact, the group $\Out(G)$  is discrete. 
Thus,  the  $E_2^{p,q}$ and ${E_2^+}^{p,q}$-terms of the spectral sequences become the group cohomology groups
$H^p(\Out(G), S^{}((\g^*)^\g[3])^q)$ and $H^p(\Out^+(G), S^{}((\g^*)^\g[3])^q)$ respectively. 
Note that the center of $G$ is stable under the action by any automorphism. 
Hence, there are  canonical group morphisms $\Out(G)\gto \Out(Z(G))$ and $\Out^+(G) \gto \Out^+(Z(G))$.

\smallskip

First assume that $G\cong Z(G)\times G'$, where $Z(G)=S^1\times \cdots \times S^1$ is a torus of dimension $n$ and $G'=[G,G]$ is semi-simple. 
Then the canonical map $\Out(G)\gto \Out(Z(G))$ has an obvious section $\Out(Z(G))\gto \Out(G)$   given by $\phi \mapsto \phi \times \id_{G'}$.
Since $G'$ is the commutator subgroup of $G$, it is also stable under automorphisms.
It follows that $\Out(G)\cong GL(n,\Z)\times \Out(G')$ and $\Out^+(G)\cong SL(n,\Z)\times \Out^+(G')$ since $\Aut(Z(G))\cong GL(n,\Z)$. 
We now need to find out the $\Out(G)$ and $\Out^+(G)$-actions  on $H^\com([Z(G)\gto 1])\cong S^{} ((\g^*)^\g[3])^q$. 
If $t_1,\dots, t_n$ are coordinates on $Z(G)$, then $(\g^*)^\g\cong \R d t_1\oplus \dots \oplus \R d t_n$ and, according to Proposition~\ref{lm:A}, the elements  $I(d t_1)$, $\dots$, $I(d t_n) \in \Omega^1(S^1\times \cdots \times S^1)\cong \Omega^3([Z(G)\gto 1])$ form a basis of $H^\com([Z(G)\gto 1])$. It follows that the $\Out(G)$ and $\Out^+(G)$-actions  on $H^\com([Z(G)\gto 1])$ reduce to the standard   $GL(n,\Z)$ and $SL(n,\Z)$-actions on the vector space $\R d t_1\oplus \dots \oplus \R d t_n$. 
Since $\Out(G')$ and $\Out^+(G')$ are finite and act trivially on  $H^\com([Z(G)\gto 1])$, the spectral sequences~\ref{eq:E+} and~\ref{eq:E} follow from the K\"unneth formula.  

\smallskip

In general, since $G$ is compact, it is isomorphic to the quotient $(G\cong Z\times G')/\Delta$, where $Z$ is the connected component of the center $Z(G)$ and $\Delta=Z(G)\cap G'$ is finite central. 
Let $\tilde{G}'$ be the universal cover of $G'$, which is a compact Lie group, and $p:Z\times \tilde{G}'\gto G$ be the covering of $G$ given by the composition $Z\times \tilde{G}'\gto Z\times G'\gto G$. 
Let $f\in \Aut(G)$, then $f\circ p: Z\times \tilde{G}'\gto G$ is a Lie group morphism. 
There is a unique lift 
\eqn
\xymatrix@C=28pt{Z\times \tilde{G}' \ar@{.>}[r]^{\tilde{f}} \ar[d]_{p}& Z\times \tilde{G}' \ar[d]^{p}\\ G \ar[r]^{f} & G}
\eneqn
of the map $f\circ p: Z\times \tilde{G}'\gto G$ into a map $\tilde{f}: Z\times \tilde{G}'\gto Z\times \tilde{G}'$ preserving the unit. 
Indeed, for this, it is sufficient to check that   $\big(f_*\circ p_*\big)(\pi_1(Z\times \tilde{G}')) \subset p_*(\pi_1(Z\times \tilde{G}'))$. 
  Clearly $p_*(\pi_1(Z\times  \tilde{G}'))\cong p_*(\Z^n)$ is the non-torsion part of $\pi_1(G)$. 
It is thus stable by any automorphism, therefore by $f_*:\pi_1(G)\gto \pi_1(G)$. 
Since $p$ is a group morphism and $\tilde{f}\in \Aut(Z\times \tilde{G}')$,  it follows that any automorphism of $G$ lifts uniquely into an automorphism of $Z\times \tilde{G}'$. 
We are thus back to the previous case.

\medskip

By the above discussions, we already know that the action of $SL(n,\Z)$ and $GL(n,\Z)$ on $(\g^*)^\g\cong \R^n$ is the standard one. 
Since the symmetric algebra on odd generators is isomorphic to an exterior algebra,  $E_2^{0,q}$ and ${E_2^+}^{0,q}$ are respectively isomorphic to $\Lambda^k \big( \R^n\big)^{GL(n,\Z)}$ (as  a $GL(n,\Z)$-module) and $\Lambda^k\big( \R^n\big)^{SL(n,\Z)}$ (as a $SL(n,\Z)$-module). 
Furthermore, if $q\neq 3k$,  $E_2^{0,q}$ and ${E_2^+}^{0,q}$ vanish for degree reasons.
 In particular, the  $GL(n,\Z)$-action is trivial for $k=0$ and, for $k=n$, it reduces to the multiplication by the determinant on $\Lambda^n\big( \R^n\big)\cong \R$.
 For $0<k<n$, $SL(n,\Z)$ (and thus $GL(n,\Z)$ too) has no fixed points in $\Lambda^k\big( \R^n\big)$ besides $0$. The last assertion follows.
\end{proof}

In general, the description of the group cohomology of $GL(n,\Z)$ and $SL(n,\Z)$ with arbitrary coefficients for general $n$ is still an open question unless $n\leq 4$ (for instance see~\cite{Sou, Hor}). 
\begin{corollary} Let $G$ be a compact Lie group. Assume that $n=\dim(\g^\g)\leq 3$. 
$$H^{p}([G\gto \Aut^+(G)])\cong \left\{\begin{array}{l} \R \quad \mbox{ if } p=0,\,3n \\
0 \quad \mbox{ otherwise}.\end{array} \right. $$ $$H^p([G\gto \Aut(G)])\cong \left\{\begin{array}{l} \R \quad \mbox{ if } p=0\\
0 \quad \mbox{ if } p>0.\end{array} \right.
$$
\end{corollary}
\begin{proof}
If $n=0$, it reduces to Proposition~\ref{ex:ssimple}.
For $n=1$, $GL(1,\Z)\cong \Z/2\Z$ and $SL(1,\Z)=\{1\}$. 
The spectral sequences of Theorem~\ref{th:main} are concentrated in bidegrees $(0,0)$ and $(0,3)$, hence collapse. 

\smallskip

For $n=2$, $SL(2,Z)$ is an amalgamated sum  $\Z/4\Z*_{\Z/2\Z}\Z/6\Z$ over a tree~\cite{Se/Trees}. 
For any $SL(2,\Z))$-module $M$, the action of $SL(2,\Z)$ on this tree yields an exact sequence
$$\dots  \gto  H^i(\Z/4\Z,M) \oplus H^i(\Z/6\Z,M) \gto H^i(\Z/2\Z,M)\gto H^{i+1}(SL(2,\Z),M)\gto \dots$$
Since the cohomology of a finite group acting on an $\R$-vector space vanishes in positive degrees,
 the only non trivial terms in the spectral sequence ${E_2^+}^{p,q}$ are for $p=0$. 
It follows that the spectral sequence collapses and the result is given by Equation~\eqref{eq:E0+} in Theorem~\ref{th:main}.
 A similar computation gives the result for $GL(2,\Z)\cong SL(2,\Z)\rtimes \Z/2\Z$. 

\smallskip

For $n=3$, one uses the fundamental domain  introduced by Soul\'e in~\cite{Sou}.
 Let $M^q$ be the $SL(3,\Z)$-module $S^{q}((\g^*)^\g[3])\cong \Lambda^q(\R^3)$ ($q=0\dots 3$).
 Since $M^0$ and $M^3$ are isomorphic to $\R$ with trivial action, the  groups $H^{p>0}(SL(3,\Z),M^q)$ are trivial for $q=0,3$. Now assume $q=1$ or $q=2$. 
The group $SL(3,\Z)$ acts by conjugation on the projective space of symmetric positive definite $3\times 3$ matrices. 
Let $D_3$ be the subset of such matrices whose diagonal coefficients are all the same. 
The orbit $X_3=D_3\cdot SL(3,\Z)$ of $D_3$ under $SL(3,\Z)$ is a homotopically trivial triangulated space of dimension 3~\cite{Sou}. Let $\Sigma_i$ be the set of equivalence classes of cells of dimension $i$ modulo the $SL(3,\Z)$-action.
 For  $\sigma \in \Sigma_i$, we denote by $SL(3,\Z)_\sigma$ the stabilizer of the cell $\sigma$ and $M^q_\sigma$ for $M^q$ endowed with the induced action of $SL(3,\Z)_\sigma$ twisted by the orientation character $SL(3,\Z)_\sigma\gto \{\pm 1\}$.  
There  is a spectral sequence $E_1^{i,j}=\bigoplus_{\sigma \in \Sigma_i} H^j\big(SL(3,\Z)_\sigma,M^q_\sigma\big)$ converging to $H^{i+j}\big(SL(3,\Z),M^q \big)$ (see~\cite{Bro} Section VII.7). 
The stabilizers $SL(3,\Z)_\sigma$ are described in~\cite{Sou} Theorem 2.
 They are all finite. 
Thus the spectral sequence reduces to $E_1^{i,0}=\bigoplus_{\sigma \in \Sigma_i} \big(M^q_\sigma\big)^{SL(3,\Z)_\sigma}$.
 Direct inspection using Theorem 2 in~\cite{Sou}  shows that  $E_1^{i\leq 1,0}=0$, $E_1^{3,0}\cong \big(M^q\big)^{4}$ and $$E_2^{2,0}\cong \big(M^q\big)^{4}\oplus \big({M^q}^{A}\big)^{3}\oplus  \big({M^q}^{B}\big)\oplus \big({M^q}^{C}\big)^{2}$$
 where $A$, $B$, $C$ are respectively the matrices {\small$$\left(\begin{array}{ccc} 0 & -1 & 0\\ -1 & 0 & 0\\ 0 &0 & -1 \end{array}\right), \quad \left(\begin{array}{ccc} -1 & 0 & 0\\ 0 & 0 & -1\\ 0 &-1 & 0 \end{array}\right),\quad \left(\begin{array}{ccc} -1 & 0 & 0\\ 0 & 0 & 1\\ 0 &1 & 0 \end{array}\right).$$}
The term $d_1$ of the spectral sequences is described in~\cite{Bro} Section VII.8. 
In our case, since the stabilizers of cells of dimension $3$ are trivial, the differential $d_1$ is induced by the inclusions $\big(M_\sigma^q\big)^{SL(3,\Z)_\sigma} \hookrightarrow M_\tau^q$ for  each 3-dimensional cell $\tau \in \Sigma_3$ with $\sigma \subset \tau$ a subface of dimension 2. 
It follows that $E_2^{i,j}\cong 0$. Hence the result follows for $[G\gto \Aut^+(G)]$. 
The case for $[G\gto \Aut(G)]$ follows  using the K\"unneth formula since $GL(3,\Z)\cong SL(3,\Z)\times \Z/2\Z$.  
\end{proof}

\begin{remark}
For $n=\dim((\g^*)^\g)=4$, it should be possible to compute explicitly $H^\com([G\gto \Aut^+(G)])$ and $H^\com([G\gto \Aut(G)])$ using Theorem~\ref{th:main} and the techniques and results of~\cite{Hor}. 
For $n=5,6$, the results of~\cite{EGS} suggest that the cohomology  groups $H^\com([G\gto \Aut^+(G)])$  and  $H^\com([G\gto \Aut(G)])$ should be non trivial. 
For larger $n$, it seems a difficult question to describe explicitly the spectral sequences of Theorem~\ref{th:main}.
\end{remark}

\end{document}